\documentclass{amsart}

\newtheorem{theorem}{Theorem}[section]

\newtheorem{remark}[theorem]{Remark}

\numberwithin{equation}{section}

\def\x#1{(\ref{#1})}
\def\R{{\mathbb R}}

\def\be{\begin{equation}}
\def\ee{\end{equation}}
\def\ba{\begin{array}}
\def\ea{\end{array}}
\def\bea{\begin{eqnarray}}
\def\eea{\end{eqnarray}}
\def\beaa{\begin{eqnarray*}}
\def\eeaa{\end{eqnarray*}}
\def\hh{\!\!\!\!}
\def\EQ{\hh & = & \hh}

\def\nn{\nonumber}

\def\ifl{\iffalse}
\def\lb{\label}
\usepackage{color}
\usepackage{graphicx}

\begin{document}

\title[]
{Constant vorticity geophysical waves with centripetal forces and at arbitrary latitude}

\author[Jifeng Chu\quad and \quad Yanjuan Yang]
{Jifeng Chu$^1$ \quad and\quad  Yanjuan Yang$^2$}
\address{$^1$Department of Mathematics, Shanghai Normal University, Shanghai 200234, China}
\address{$^2$Department of Mathematics, Hohai University, Nanjing 210098, China}
\email{jifengchu@126.com; jchu@shnu.edu.cn (J. Chu)}
\email{yjyang90@163.com (Y. Yang)}
\thanks{Jifeng Chu was supported by the National Natural Science Foundation of China (Grants No. 11671118 and No. 11871273).
Yanjuan Yang was supported by the Fundamental Research Funds for the Central Universities
(Grant No. 2017B715X14) and the Postgraduate Research and Practice Innovation Program of
Jiangsu Province (Grant No. KYCX17$_{-}$0508).}


\begin{abstract}
We consider three-dimensional geophysical flows at arbitrary latitude and with constant vorticity
beneath a wave train and above a flat bed in the $\beta$-plane approximation with centripetal
forces. We consider the $f$-plane approximation as well as the $\beta$-plane approximation.
For the $f$-plane approximation, we prove that there is no bounded solution. For the $\beta$-plane approximation,
we show that the flow is necessarily irrotational and the free surface is necessarily flat if
it exhibits a constant vorticity. Our results reveal some essential differences from those results
in the literature, due to the presence of centripetal forces. Moreover, for the case exhibiting the surface tension, we prove that there are no flows
exhibiting constant vorticity.
\end{abstract}

\maketitle

\section{Introduction}

In this paper, we focus on geophysical ocean waves in which both Coriolis and
centripetal effects of the Earth's rotation play a significant role. In recent
years, the mathematical analysis of geophysical flows \cite{cb, p} has attracted much attention
for their wide applications (see the references \cite{bm, c-jgr, c-jpo, c2014, cj-1, cj-2, cj-3, h-jmbf, im, ic, wh} for the flows in the equatorial region
and \cite{cdy-dcds, cdy} for the flows at arbitrary latitude). However, in most existed results,
centripetal forces are typically neglected because they are relatively much smaller than Coriolis forces.
Recently Henry in \cite{h-jfm} showed in a remarkable way that the relatively small-scale centripetal
force plays a central role in facilitating the admission of a wide range of constant underlying currents
in studying the exact solution for the equatorially nonlinear waves in the $\beta$-plane approximation
 and with centripetal forces.
Later, an explicit three-dimensional nonlinear solution for  geophysical waves propagating at arbitrary latitude in the $\beta$-plane approximation with centripetal forces was presented in \cite{cdy}.

Compared with large studies on equatorial water waves, the study on the non-equatorial
waves seems much fewer. Besides the work \cite{cdy} mentioned above, 
an extension of the exact solution \cite{hh-dcds} for equatorial waves in the $f$-plane approximation to the cases
at arbitrary latitude and in the presence of a constant underlying background current was presented in \cite{fgx}.
 A $\beta$-plane approximation at arbitrary latitude in the presence of an underlying current and a Gerstner-like solution to this
problem was very recently provided in \cite{cdy-dcds}.

Vorticity is adequate in describing the motion of both equatorial and non-equatorial flows.
The nonzero vorticity serves as a tool
for describing interactions of waves with non-uniform currents. From the history perspective,
the mathematical theory of rotational water waves was original started by Gerstner in the beginning of
the 19th century \cite{gf}, in which an explicit family of periodic travelling waves with non-zero vorticity
was constructed using Lagrangian coordinates. In recent works \cite{c-2011, ck, csv, cv, cn, ic, m-3, s}, the assumption of nonzero constant vorticity
has been assumed, which is the simplest rotational setting and corresponds to a uniform current. Although such an
assumption is for physical viewpoints (see the discussion in \cite{ligh}), the main consideration lies on more
convenient in mathematical analysis, for example, constant vorticity flows have
the advantage that their velocity field consists of harmonic functions (see the modern discussions in \cite{csv, cv}).
The importance of the vorticity in
the realistic modeling of ocean flows is highlighted in the very recent papers \cite{bm, cj-4, m-jmfm, m-ampa}.
See \cite{cew, cim, cs, j, m-jmfm, wah, wh} and the monograph \cite{c-book}for
more results on  rotational water waves.

Among the results on vorticity in the literature, a feature is to determine the
dimensionality of the flow. From the mathematical perspective, the study of the
two-dimensionality for the rotational flow was started by the work \cite{c-2011},
in which Constantin showed that a free surface water flow of constant nonzero
vorticity beneath a wave train and above a flat bed must be two-dimensional and
the vorticity must have only one nonzero component which points in the horizontal
direction orthogonal to the direction of wave propagation. After \cite{c-2011}, more
results along this line has been obtained in different settings. In the presence of
Coriolis forces, Martin in \cite{m-ampa} proved the two-dimensionality of the equatorial
flows in the $f$-plane approximation, and it was found that there is a striking difference
between the geophysical flows and the classical gravity flows, that is, the two-dimensionality
holds even if the vorticity vector vanishes due to the presence of Coriolis forces. Martin also
proved in \cite{m-JFM} and \cite{m-2019} that for the equatorial and non-equatorial flows in
the $\beta$-plane approximation, the only flow exhibiting a constant vorticity vector is the
stationary flow with vanishing velocity field and flat surface. Very recently, the authors
\cite{cy} obtained several results much different from \cite{m-JFM}, and we show that
the equatorial flow is necessarily irrotational, the free surface is necessarily flat, and
possess non-vanishing horizontal velocity field if it exhibits a constant vorticity, owing
to the presence of centripetal forces.

The aim of this paper is to show that, assuming that the non-equatorial flows admit
a constant vorticity vector,  the centripetal force can lead to a better outcome,
especially compared with the existed results without the centripetal term. Both $f$-plane
approximation and $\beta$-plane approximation are studied. We will extend the results
in \cite{cy} to the flows at arbitrary latitude. In particular, for the $f$-plane
approximation, we prove that there is no bounded solution, while for the $\beta$-plane approximation,
we show that the flow is necessarily irrotational and the free surface is necessarily flat if
it exhibits a constant vorticity. Moreover, for the case exhibiting the surface tension, we
prove that there are no flows exhibiting constant vorticity.

\section{Preliminary}

We recall the following governing equations derived by Constantin and Johnson in \cite{cj-2} for
geophysical fluid dynamics in the cylindrical coordinates
\begin{equation}\left\{ \begin{array}{ll}
u_t+uu_x+\frac{vu_\phi}{R+z}+wu_z+2\Omega(w\cos\phi-v\sin\phi) =-\frac{1}{\rho}P_x, \\
v_t+uv_x+\frac{vv_\phi}{R+z}+\frac{wv}{R+z}+2\Omega u\sin\phi+(R+z)\Omega^2\sin\phi\cos\phi =-\frac{1}{\rho}\frac{P_\phi}{R+z},\\
w_t+uw_x+\frac{vw_\phi}{R+z}+ww_z-\frac{v^2}{R+z}-2\Omega u\cos\phi-(R+z)\Omega^2\cos^2\phi=-\frac{1}{\rho}P_z-g,\end{array} \right.\nonumber\end{equation}
together with the equation of incompressibility
\be\lb{ein}u_x+\frac{1}{R+z}v_\phi+\frac{1}{R+z}\frac{\partial}{\partial z}[(R+z)w]=0.\nonumber\ee
Here the origin in the cylindrical coordinates is located at the centre of the Earth, $x$-axis with the positive $x$-direction going from west to east, $\phi$ is the angle of latitude and $z=r-R$ is the variation in the locally vertical direction of the radial variable from the Earth's surface,
 $(u,v,w)$ is the fluid velocity field, $P$ is the pressure,
$\rho$ is the water's density,  $t$ is the time, $g$
is the standard gravitational acceleration at the Earth's surface and  $\Omega=7.29\times 10^{-5}$ rad/s  is the rotational speed
of the Earth and $R=$6378 km is the radius of the Earth.

The Coriolis parameters, defined by:
\be f=2\Omega\sin\phi,\quad \hat{f}=2\Omega \cos\phi,\nonumber\ee
 depend on the variable latitude $\phi$.
At the Equator
$f=0$, $\hat{f}=2\Omega$.
For water waves propagating zonally in a relatively narrow
ocean strip less than a few degrees of latitude wide,
it is adequate to use the $f$- or $\beta$-plane approximations.
Within the $f$-plane approximation  the Coriolis parameters  are treated  as constants,
and in terms of the Cartesian coordinate system $(x,y,z)$, we obtain
the governing equations
\begin{equation}\lb{Euler-1-f}\left\{ \begin{array}{ll}
u_t+uu_x+vu_y+wu_z+\hat{f}w-fv=-\frac{1}{\rho}P_x, \\
v_t+uv_x+vv_y+wv_z +fu+\frac{\hat{f}^2}{4}y+\frac{\hat{f}f}{4}R=-\frac{1}{\rho}P_y,\\
w_t+uw_x+vw_y+ww_z-\hat{f} u-\frac{\hat{f}^2}{4}R=-\frac{1}{\rho}P_z-g.
\end{array} \right.\end{equation} Within the $\beta$-plane approximation, we consider that, at the fixed latitude $\phi$,
$\hat{f}$ is constant and  $f$ has  a linear variation with the latitude. Defining $y=R\alpha$ and retaining only terms of linear order  in the expansion of  $\sin(\phi+\alpha)$, this linear variation has the form $f+\beta y$,
 with
\begin{align} \beta=\frac{\hat{f}}{R}=\frac{2\Omega\cos\phi}{R}.\nonumber\end{align}
Thus we get the following $\beta$-plane  approximation   equations for geophysical
fluid dynamics with centripetal terms:
\begin{equation}\lb{Euler-1}\left\{ \begin{array}{ll}
u_t+uu_x+vu_y+wu_z+\hat{f}w-(f+\beta y)v=-\frac{1}{\rho}P_x, \\
v_t+uv_x+vv_y+wv_z +(f+\beta y)u+\frac{\hat{f}^2}{4}y+\frac{\hat{f}f}{4}R=-\frac{1}{\rho}P_y,\\
w_t+uw_x+vw_y+ww_z-\hat{f} u-\frac{\hat{f}^2}{4}R=-\frac{1}{\rho}P_z-g.
\end{array} \right.\end{equation}
In both cases, we have the condition of incompressibility
\be\lb{massc}u_x+v_y+w_z=0.\ee

We will consider regular wave trains of water waves propagating steadily in the direction of the horizontal $x$-axis, $L$-periodic in the variable $x$, and presents no variation in the $y$-direction. The fluid domain
is bounded below by the impermeable flat bed $z=-d$, and above by the free surface $z=\eta(x-ct)$, where $\eta$ gives the wave profile with
the zero mean $\int_{0}^{L}\eta(s)ds=0$ and $c>0$ is the wave speed.
We assume that the wave crest is located at $x=0$, and thus obviously we know $\eta(0)>0$.

Complementing the equations of motion are the boundary conditions, of which
\be\lb{kbc-3}P=P_{atm}\quad {\rm on}\quad z=\eta(x-ct),\ee
with $P_{atm}$ being the constant atmospheric pressure, decouples the motion
of the water from that of the air.  In addition to \x{kbc-3}, we
have the kinematic boundary conditions
\be\lb{kbc-1}w=(u-c)\eta_x\quad {\rm on}\quad z=\eta(x-ct),\ee
and
\be\lb{kbc-2}w=0\quad {\rm on}\quad z=-d.\ee
In the presence of surface tension, \x{kbc-3} is replaced by
\be\lb{kbc-4}P=P_{atm}-\sigma \frac{\eta_{xx}}{(1+\eta^2_x)^{3/2}}\quad {\rm on}\quad z=\eta(x-ct),\ee
where the constant $\sigma>0$ is the surface tension coefficient, and we assume that $\eta\in C^2(\R^2)$ in \x{kbc-4}.

The vorticity vector $\Upsilon$ is defined as the curl of the velocity field $\mathbf{u}=(u,v,w)$:
\be\lb{vorticity}\mathbf\Upsilon=(\Upsilon_1,\Upsilon_2,\Upsilon_3)=(w_y-v_z,u_z-w_x,v_x-u_y).\ee
In this paper, we assume that the vorticity vector is constant and satisfies 
\be\lb{no-0}\Upsilon_2+\hat{f}\neq0,\quad \Upsilon_3+f\neq0,\ee
which are reasonable since the magnitude of the equatorial undercurrent's relative vorticity
is much larger than that of the planetary vorticity (see the discussions in \cite{c-grl}).

\section{$f$-plane approximation}

In this Section, we consider the $f$-plane approximation, which corresponds to the governing equations \x{Euler-1-f}
with the conditions {\rm\x{massc}-\x{kbc-2}}. The main result of this Section reads as follows.

\begin{theorem}\lb{thm-4} Assume that the vorticity vector $\mathbf\Upsilon$ is
constant and  satisfies {\rm \x{no-0}}. Then there is no bounded solution to the
equations {\rm \x{Euler-1-f}} with {\rm\x{massc}-\x{kbc-2}}.
\end{theorem}
\begin{proof}
It is easy to verify that the constant vorticity vector $\mathbf\Upsilon$ satisfies the equation
\[(\mathbf{\Upsilon}\cdot\nabla)\mathbf{u}+\hat{f}(u_y,v_y,w_y)+f(u_z,v_z,w_z)=0,\]
which is equivalent to the following three equalities
\be\lb{v-f-1}\Upsilon_1 u_x+(\Upsilon_2+\hat{f})u_y+(\Upsilon_3+f)u_z=0,\ee
\be\lb{v-f-2}\Upsilon_1 v_x+(\Upsilon_2+\hat{f})v_y+(\Upsilon_3+f)v_z=0,\ee
\be\lb{v-f-3}\Upsilon_1 w_x+(\Upsilon_2+\hat{f})w_y+(\Upsilon_3+f)w_z=0.\ee
From \x{v-f-3}, we know that $w$ is constant in the direction of the vector $(\Upsilon_1, \Upsilon_2+\hat{f}, \Upsilon_3+f)$, which is not parallel to the flat bed $z=-d$ due to the condition $\Upsilon_3+f\neq0$. Using the kinematic boundary condition \x{kbc-2}, we obtain that $w=0$ throughout the fluid domain. Thus, we obtain from \x{vorticity} that
\[u_z=\Upsilon_2\quad\hbox{and}\quad v_z=-\Upsilon_1.\]
From the above relations, we can infer that there exist two functions $\hat{u}=\hat{u}(x,y,t)$, $\hat{v}=\hat{v}(x,y,t)$
such that
\be\lb{uu}u(x,y,z,t)=\hat{u}(x,y,t)+\Upsilon_2 z,\ee
\be\lb{vv}v(x,y,z,t)=\hat{v}(x,y,t)-\Upsilon_1 z,\ee
for all $x,y,z,t$ with $-d\leq z\leq\eta(x-ct)$. Due to \x{massc}, the functions $\hat{u}$ and
$\hat{v}$ satisfy the equation
\[\hat{u}_x+\hat{v}_y=0,\]
which admits us to choose a function $\psi=\psi(x,y,t)$ satisfying
\be\lb{str-f}\hat{u}=\psi_y\quad\hbox{and}\quad\hat{v}=-\psi_x.\ee
Consequently, from the equations \x{v-f-1}-\x{v-f-2}, we deduce that
\begin{equation}\lb{vorticity-stream}\left\{ \begin{array}{ll}
\Upsilon_1 \psi_{xy}+(\Upsilon_2+\hat{f})\psi_{yy}+(\Upsilon_3+f)\Upsilon_2=0, \\
\Upsilon_1 \psi_{xx}+(\Upsilon_2+\hat{f})\psi_{xy}+(\Upsilon_3+f)\Upsilon_1=0.
\end{array} \right.\end{equation}
We also obtain from the definition of $\Upsilon_3$ that
\be\lb{s-v-3}\psi_{xx}+\psi_{yy}=-\Upsilon_3.\ee
Using the relations \x{vorticity-stream} and \x{s-v-3}, we have
\[\psi_{xx}=\frac{\hat{\Upsilon}_2(f\Upsilon_2-\hat{f}\Upsilon_3)
  -\Upsilon_1^2\Upsilon_3-f\Upsilon_1^2}{\Upsilon^2_1+\hat{\Upsilon}^2_2}:=A,\]
\[\psi_{xy}=-\frac{\Upsilon_1(\Upsilon_2\hat{\Upsilon}_3+f\hat{\Upsilon}_2)}{\Upsilon^2_1+\hat{\Upsilon}^2_2}:=B,\]
\[\psi_{yy}=\frac{f\Upsilon^2_1-\hat{\Upsilon}_2\hat{\Upsilon}_3\Upsilon_2}{\Upsilon^2_1+\hat{\Upsilon}^2_2}:=C.\]
where
\[\hat{\Upsilon}_2=\Upsilon_2+\hat{f},\quad \hat{\Upsilon}_3=\Upsilon_3+f.\]
Therefore, there exist functions $d(t),e(t),g(t)$ such that
\[\psi(x,y,t)=\frac{1}{2}Ax^2+Bxy+\frac{1}{2}Cy^2+d(t)x+e(t)y+g(t).\]
By \x{str-f}, we find that
\[\hat{u}(x,y,t)=Bx+Cy+e(t),\]
\[\hat{v}(x,y,t)=-Ax-By-d(t).\]
Since the functions $\hat{u}$ and $\hat{v}$ are bounded, we can infer that
\[A=B=C=0.\]

Now, we claim that $\Upsilon_1=0$. On the contrary, we assume that $\Upsilon_1\neq0$. Since $B=0$, we conclude that
\[\Upsilon_2\hat{\Upsilon}_3+f\hat{\Upsilon}_2=0,\]
and thus
\[\hat{\Upsilon}_2(\Upsilon_2\hat{\Upsilon}_3+f\hat{\Upsilon}_2)=0.\]
Using the fact $C=0$, the above equation becomes
\[f(\Upsilon^2_1+\hat{\Upsilon}^2_2)=0,\]
which is impossible. Therefore  $\Upsilon_1=0$.

Since $C=0$, we can infer that $\hat{\Upsilon}_2\hat{\Upsilon}_3\Upsilon_2=0$, owing to \x{no-0}, we can conclude that
$\Upsilon_2=0$. Moreover,  we can obtain from $A=0$ that $\hat{f}\hat{\Upsilon}_2\Upsilon_3=0$. Because $\hat{\Upsilon}_2\neq0$ and $\hat{f}\not\equiv0$, we derive that $\Upsilon_3=0$.

From \x{uu} and \x{vv}, we obtain that
\[u(x,y,z,t)=\hat{u}(x,y,t)=e(t)\doteq u(t),\]
\[v(x,y,z,t)=\hat{v}(x,y,t)=-d(t)\doteq v(t),\]
which mean that $u,v$ are only dependent of $t$.
Moreover, from \x{Euler-1-f}, we obtain
\[\lb{PxPyPz}\left\{ \begin{array}{ll}
P_x=-\rho[u'(t)-fv(t)], \\
P_y=-\rho\Big[v'(t)+fu(t)+\frac{\hat{f}^2}{4}y+\frac{\hat{f}f}{4}R\Big],\\
P_z=\rho \Big[\hat{f} u(t)+\frac{\hat{f}^2}{4}R-g\Big].
\end{array} \right.\]
Therefore, the pressure can be given as
\begin{eqnarray*}
P(x,y,z,t)&=&-\rho[u'(t)-fv(t)]x-\rho\Big[\Big(v'(t)+fu(t)+\frac{\hat{f}f}{4}R\Big)y+\frac{\hat{f}^2}{8}y^2\Big]\\
&~&+\rho \Big[\hat{f} u(t)+\frac{\hat{f}^2}{4}R-g\Big]z+\bar{p}(t).
\end{eqnarray*}
Now the kinematic boundary condition \x{kbc-3} becomes
\begin{eqnarray*}\lb{kbc-becomes}P_{atm}&=&-\rho[u'(t)-fv(t)]x-\rho\Big[\Big(v'(t)+fu(t)+\frac{\hat{f}f}{4}R\Big)y
+\frac{\hat{f}^2}{8}y^2\Big]\\
&~&+\rho \Big[\hat{f} u(t)+\frac{\hat{f}^2}{4}R-g\Big]\eta(x-ct)+\bar{p}(t),\end{eqnarray*}
for all $x,y,t$. We infer from the above equation that the coefficient of $y$ must vanish, which is impossible.
Therefore, we conclude that there is no solution to the equations \x{Euler-1-f} with \x{massc}-\x{kbc-2}.
\end{proof}

\section{$\beta$-plane approximation}

In this section, we consider the $\beta$-plane approximation, which corresponds to the equations \x{Euler-1}-\x{massc} with the conditions \x{kbc-3}-\x{kbc-2}. Using \x{Euler-1} and \x{massc}, the vorticity equation becomes
\[\mathbf{\Upsilon}_t+(\mathbf{u}\cdot\nabla)\mathbf{\Upsilon}-\hat{f}(u_y,v_y,w_y)
-(f+\beta y)(u_z,v_z,w_z)+\beta(0,0,v)
=(\mathbf{\Upsilon}\cdot\nabla)\mathbf{u}.\]
For the constant vorticity, we can obtain
\[(\mathbf{\Upsilon}\cdot\nabla)\mathbf{u}+\hat{f}(u_y,v_y,w_y)+(f+\beta y)(u_z,v_z,w_z)-\beta(0,0,v)=0,\]
which is equivalent to the following equalities
\be\lb{vorticity-1}\Upsilon_1 u_x+(\Upsilon_2+\hat{f})u_y+(\Upsilon_3+f+\beta y)u_z=0,\ee
\be\lb{vorticity-2}\Upsilon_1 v_x+(\Upsilon_2+\hat{f})v_y+(\Upsilon_3+f+\beta y)v_z=0,\ee
\be\lb{vorticity-3}\Upsilon_1 w_x+(\Upsilon_2+\hat{f})w_y+(\Upsilon_3+f+\beta y)w_z-\beta v=0.\ee

\begin{theorem}\lb{lemma-1} There is no water flow exhibiting non-zero constant
vorticity vector and with a flat surface. Indeed, any flow with a flat surface and
constant vorticity vector must have the vanishing vorticity vector, that is $\mathbf{\Upsilon}=(0,0,0)$.
\end{theorem}
\begin{proof}
We can obtain from \x{vorticity} and the equation  \x{massc} that
\[\Delta w=w_{xx}+w_{yy}+w_{zz}=u_{zx}+v_{zy}+w_{zz}=(u_{x}+v_{y}+w_{z})_z=0.\]
Analogously,
\[\Delta u=\Delta v=0.\]
 Therefore, the velocity components $u,v,w$ are harmonic function within the
fluid domain. Moreover, it is obvious that all partial derivatives of $u,v,w$ are harmonic functions.
Then it follows from  \x{vorticity-1}  that
\[\Delta(y u_z)=0,\]
 which can be written as
\[y\Delta(u_z)+2u_{z y}=0,\]
from which we obtain that
\[u_{zy}=0.\]
Similarly, using the equations \x{vorticity-2} and \x{vorticity-3}, we can infer that
\[v_{zy}=0\quad\hbox{and}\quad w_{zy}=0.\]
Then, by the definitions of $\Upsilon_1$ and $\Upsilon_2$, we have
\[w_{yy}=v_{zy}=0\quad\hbox{and}\quad w_{xy}=u_{zy}=0.\]
Using the above relations,
we conclude that $w_y=f(t)$ for some function $f$, combined with the kinematic boundary condition \x{kbc-2}, we have
\[w_y=0\quad {\rm on}\quad z=-d.\]
Therefore, we conclude that
\[w_y\equiv0,\]
which implies that $v_z=-\Upsilon_1$.

Differentiating with respect to $y$ in \x{vorticity-3}, we obtain
\[\Upsilon_1 w_{xy}+(\Upsilon_2+\hat{f})w_{yy}+(\Upsilon_3+f+\beta y)w_{zy}+\beta w_z-\beta v_y=0.\]
Since $w_y\equiv0$, we can infer that
\be\lb{wzvy}w_z=v_y,\ee from which we have
\be\lb{wzz-vyz}w_{zz}=v_{yz}=(-\Upsilon_1)_y=0,\ee and
\[v_{yy}=w_{zy}=0.\]
Moreover, due to $v_{zz}=(-\Upsilon_1)_z\equiv0$ and $\Delta v=\Delta w=0$, we conclude that
\[v_{xx}\equiv0\quad\hbox{and}\quad w_{xx}\equiv0.\]
Differentiating with respect to $z$ in the equation of mass conservation \x{massc} we get
\[u_{xz}+v_{yz}+w_{zz}=0.\]
Due to \x{wzz-vyz}, we get
\[u_{xz}=0.\]
Analogously, differentiating with respect to $y$ in \x{massc}, we can obtain that
\[u_{xy}=0.\]
Now, let us differentiate with respect to $z$ in the equation \x{vorticity-1}, we have
\[\Upsilon_1 u_{xz}+(\Upsilon_2+\hat{f})u_{yz}+(\Upsilon_3+f+\beta y)u_{zz}=0,\]
which becomes
\[(\Upsilon_3+f+\beta y)u_{zz}=0,\]
since $u_{xz}=u_{yz}=0$.
Thus $u_{zz}(x,y,z)=0$ for all $(x,y,z)$ with $y\neq-\frac{\Upsilon_3+f}{\beta}$. From the continuity of $u_{zz}$, we have
\[u_{zz}=0\quad{\rm within\; the \; fluid\;domain}. \]
Recalling that $\Upsilon_2=u_z-w_x$, it is easy to obtain that
\[w_{xz}=0.\]
By the fact that $w_{xx}=w_{xy}=0$, we conclude that $w_x=a(t)$ for some function $a$. Since $w_x=0$ on the flat bed $z=-d$, we obtain
\[w_x=0\quad{\rm within\; the \; fluid\;domain},\]
which implies that $w_{zx}=0$ and $\Upsilon_2=u_z$ . Moreover, by the fact $w_{zy}=w_{zz}=0$,
we conclude that
\[w_z~{\rm is ~constant \;within\; the \; fluid\;domain}.\]
Differentiating the equation \x{vorticity-2} with respect to $x$, we have
\[\Upsilon_1 v_{xx}+(\Upsilon_2+\hat{f})v_{yx}+(\Upsilon_3+f+\beta y)v_{zx}=0.\]
Since $v_{zx}=(-\Upsilon_1)_x=0$, $v_{xx}=0$, we get
\[(\Upsilon_2+\hat{f})v_{yx}=0.\]
Using the assumption $\Upsilon_2+\hat{f}\neq0$, we deduce that
\[v_{xy}=0\quad{\rm within\; the \; fluid},\]
which, by the expression of $\Upsilon_3=v_x-u_y$, implies that
\[u_{yy}=0\quad{\rm within\; the \; fluid}.\]

Using the previous relations $u_{xy}=u_{yy}=v_{xy}=v_{yy}=0$, we can obtain from the equations \x{vorticity-1} and \x{vorticity-2} that
\be\lb{uzvz}u_{z}=v_z=0\quad{\rm throughout\; the \; flow},\ee
which yields that
\[\Upsilon_1=\Upsilon_2=0.\] Thus, equations \x{vorticity-1} and \x{vorticity-2} become
\[\hat{f} u_y=0,\quad\hbox{and}\quad\hat{f} v_y=0,\]
which allow us to conclude that
\be\lb{uyvy}u_y=v_y=0.\ee
Notice that \x{wzvy} holds, so we have
\be\lb{w-z}w_z=0\quad \mbox{within~the~fluid}.\ee
Due to $\Upsilon_1=\Upsilon_2=0$ and $w_x=w_y=0$, the equation \x{vorticity-3} can be simplified to
\[(\Upsilon_3+f+\beta y)w_z=\beta v.\]
By \x{w-z}, we obtain
\[v=0\quad{\rm within\; the \; fluid}.\]
Therefore, combined with \x{uyvy} we have
\[\Upsilon_3=v_x-u_y=0.\] Now the proof is finished.
\end{proof}

\begin{theorem}\lb{thm-2} Assume that the vorticity vector $\mathbf\Upsilon$ is constant and satisfies $\Upsilon_2+\hat{f}\neq0$. Then the only bounded solution to the equations {\rm\x{Euler-1}-\x{massc}} with the conditions {\rm\x{kbc-3}-\x{kbc-2}} is the one with flat surface, velocity field and the pressure given as
\[(u,v,w)=(-\frac{\hat{f}^2}{4\beta},0,0),\]
\be\lb{pxyz}P(x,y,z,t)=\rho \Big[-\frac{\hat{f}^3}{4\beta}+\frac{\hat{f}^2}{4}R-g\Big](z-\eta_0)+P_{atm},\ee
where $\eta_0$ is a constant.
\end{theorem}
\begin{proof}
From the proof of Theorem \ref{lemma-1}, we can deduce that
\[w=0\quad{\rm within\; the \; fluid\;domain},\]
since $w_z=0$ and $w=0$ on the flat bed $z=-d$. In addition, we have shown that $v=0$. Thus we only need to find the horizontally velocity $u$ for the velocity field.

From the equation \x{massc} and the fact $v_y=w_z=0$,
we obtain that $u_x=0$. Going back to \x{uzvz} and \x{uyvy}, we conclude that
\[u(x,y,z,t)=b(t)\]
for some function $b$.

Note that $(u,v,w)=(b(t), 0 , 0)$, the Euler equations \x{Euler-1} become
\begin{equation}\begin{split}\lb{Euler-P}\left\{ \begin{array}{ll}
P_x=-\rho b'(t), \\
P_y=-\rho \Big[(f+\beta y)b(t)+\frac{\hat{f}^2}{4}y+\frac{\hat{f}f}{4}R\Big],\\
P_z=\rho \Big[\hat{f} b(t)+\frac{\hat{f}^2}{4} R-g\Big].
\end{array} \right.\end{split}\end{equation}
Therefore, the pressure can be given as
\[P(x,y,z,t)=-\rho b'(t)x-\rho\Big[fb(t)y+\frac{ \beta b(t)}{2}y^2+\frac{\hat{f}^2}{8}y^2+\frac{\hat{f}f}{4}Ry\Big]
+\rho \Big[\hat{f} b(t)+\frac{\hat{f}^2}{4} R-g\Big]z+p(t).\]
Now the kinematic boundary condition \x{kbc-3} becomes
\bea\lb{P-z-condition}&& P_{atm}+\rho\Big[fb(t)y+\frac{ \beta b(t)}{2}y^2+\frac{\hat{f}^2}{8}y^2+\frac{\hat{f}f}{4}Ry\Big]\nn\\
\EQ-\rho b'(t)x+\rho \Big[\hat{f} b(t)+\frac{\hat{f}^2}{4} R-g\Big]\eta(x-ct)+p(t),\eea
for all $x,y,t$. We infer from the above equation that the coefficient of $y$ must vanish, which means that
$b(t)=-\frac{\hat{f}^2}{4\beta}$. Now the equality \x{P-z-condition} simplifies to
\[P_{atm}=\rho \Big[-\frac{\hat{f}^3}{4\beta}+\frac{\hat{f}^2}{4} R-g\Big]\eta(x-ct)+p(t)\quad\hbox{for\;all}\quad x, t,\]
which is only possible if both functions $p$ and $\eta$ are constants $p_0$, $\eta_0$.
Thus the pressure function can be given as the form  \x{pxyz}.
\end{proof}

\begin{remark}\lb{r1}The above result is also true for the case that the fluid domain bounded below by the flat bed $z=-d$ and above by the free surface $z=\eta(x,y,t)$ {\rm(}not the wave trains{\rm)}. In fact, the velocity field, the pressure and the free surface given by
\begin{equation*}\begin{split}\lb{solution}\left\{ \begin{array}{ll}
(u,v,w)=(-\frac{\hat{f}^2}{4\beta},0,0), \\
P(x,y,z,t)=\rho \Big[-\frac{\hat{f}^3}{4\beta}+\frac{\hat{f}^2}{4} R-g\Big](z-\tilde{\eta}_0)+P_{atm},\\
\eta(x,y,t)=\tilde{\eta}_0,
\end{array} \right.\end{split}\end{equation*}
{\rm(}where $\tilde{\eta}_0$ is a constant{\rm)} is the only solution satisfying the equations {\rm\x{Euler-1}}-{\rm\x{massc}} with the boundary conditions
\[P=P_{atm}\quad{\rm on}\quad z=\eta(x,y,t),\]
\[w=\eta_t+u\eta_x+v\eta_y\quad{\rm on}\quad z=\eta(x,y,t),\]
\[w=0\quad{\rm on}\quad z=-d.\]
\end{remark}

The following Remark presents much difference between our results and the results for the flows without centripetal effects and in the $\beta$-plane
approximation.
\begin{remark}\lb{r2}The tuple $(u, v, w, P, \eta)$ representing the velocity field, the pressure and the free surface given by
\begin{equation*}\begin{split}\lb{solution}\left\{ \begin{array}{ll}
(u,v,w)=(0,0,0), \\
P(x,y,z,t)=-\rho g(z-\bar{\eta}_0)+P_{atm},\\
\eta(x,y,t)=\bar{\eta}_0,
\end{array} \right.\end{split}\end{equation*}
{\rm(}where $\bar{\eta}_0$ is a constant{\rm)} is the only flow which satisfies the Euler equations {\rm(}without the centripetal forces{\rm)}
\begin{equation*}\lb{Euler-r2}\left\{ \begin{array}{ll}
u_t+uu_x+vu_y+wu_z+\hat{f} w-(f+\beta y)v=-\frac{1}{\rho}P_x, \\
v_t+uv_x+vv_y+wv_z +(f+\beta y)u=-\frac{1}{\rho}P_y,\\
w_t+uw_x+vw_y+ww_z-\hat{f} u =-\frac{1}{\rho}P_z-g,
\end{array} \right.\end{equation*}
together with the conditions {\rm\x{massc}} and {\rm \x{kbc-3}-\x{kbc-2}}.
\end{remark}

Finally in this section, we will prove a result for  capillary-gravity waves, which correspond to the equations \x{Euler-1}-\x{massc} and the boundary conditions \x{kbc-1}-\x{kbc-4}.

\begin{theorem}\lb{thm-3} Assume that the vorticity vector is constant and also $\Upsilon_2+\hat{f}\neq0$. Then there is no bounded solution to the
equations {\rm\x{Euler-1}-\x{massc}} with {\rm\x{kbc-1}-\x{kbc-4}}.
\end{theorem}
\begin{proof}
Note that \x{Euler-P} in Theorem \ref{thm-2} still holds here. Thus, the pressure should be given as
\beaa P(x,y,z,t)\EQ-\rho b'(t)x-\rho\Big[fb(t)y+\frac{ \beta b(t)}{2}y^2+\frac{\hat{f}^2}{8}y^2+\frac{\hat{f}f}{4}Ry\Big]\\
&+&\rho \Big[\hat{f} b(t)+\frac{\hat{f}^2}{4} R-g\Big]z+p(t).\eeaa
Using the condition \x{kbc-4}, we obtain that
\beaa P_{atm}-p(t)
&=&-\rho b'(t)x-\rho\Big[fb(t)y+\frac{ \beta b(t)}{2}y^2+\frac{\hat{f}^2}{8}y^2+\frac{\hat{f}f}{4}Ry\Big]\\
&+&\rho \Big[\hat{f} b(t)+\frac{\hat{f}^2}{4} R-g\Big]\eta(x-ct)+\sigma\frac{\eta_{xx}}{(1+\eta^2_x)^{3/2}}.\eeaa
Therefore, we conclude that
\[\frac{\beta b(t)}{2}+\frac{\hat{f}^2}{8}=0,\quad fb(t)+\frac{\hat{f}f}{4}R=0.\] Then
\[b(t)=-\hat{f}^2/(4 \beta)\]
and
\be\lb{P2-thm3}P_{atm}-p(t)
=\rho \Big[-\frac{\hat{f}^3}{4 \beta}+\frac{\hat{f}^2}{4} R-g\Big]\eta(x-ct)+\sigma\frac{\eta_{xx}}{(1+\eta^2_x)^{3/2}}\ee
for all $x,t$. Notice that the function
\[x\rightarrow \frac{\eta_{x}(x)}{\sqrt{1+\eta^2_x(x)}}\]
is periodic and $\int_{0}^{L}\eta(s)ds=0$, then we obtain upon integration from $0$ to $L$ in \x{P2-thm3} that
\be\lb{P3-thm3}\rho \Big[\frac{\hat{f}^3}{4 \beta}-\frac{\hat{f}^2}{4} R+g\Big]\eta(x)=\sigma\frac{\eta_{xx}}{(1+\eta^2_x)^{3/2}}\quad\hbox{for\;all}\;x.\ee
Due to \[\frac{\hat{f}^3}{4 \beta}-\frac{\hat{f}^2}{4} R+g>0\] and $\eta(0)>0$ implies (due to the continuity of $\eta$) that $\eta(x)>0$ in a neighborhood $B_\epsilon(0)$ of $x=0$, we deduce from \x{P3-thm3} that $\eta_{xx}>0$ in $B_\epsilon(0)$, which yields
that the function $\eta$ is convex in $B_\epsilon(0)$, this contradicts the maximality of $\eta$ at the crest.
Therefore, there is no bounded solutions to   {\rm\x{Euler-1}-\x{massc}} with {\rm\x{kbc-1}-\x{kbc-4}}.
\end{proof}

\bibliographystyle{amsplain}

\end{document}